\documentclass[a4paper,10pt]{amsart}
\usepackage{amssymb,amsmath,amscd,url,graphicx,amsthm,a4,amsfonts,enumerate,color,url}
\usepackage[pdftex]{hyperref}

\usepackage{color}

\usepackage[utf8]{inputenc}

\newtheorem{defi}{Definition}[section]
\newtheorem{theo}[defi]{Theorem}
\newtheorem{lem}[defi]{Lemma}
\newtheorem{nota}[defi]{Notation}
\newtheorem{coro}[defi]{Corollary}
\newtheorem{rema}{Remark}
\newtheorem{exa}{Example}
\def\me{\mathsf{e}}
\def\mv{\mathsf{v}}
\def\mw{\mathsf{w}}
\numberwithin{equation}{section}

\title[Travelling waves for the BBM equation on networks]{Construction of exact travelling waves for the Benjamin-Bona-Mahony equation on networks}

\author{Delio Mugnolo}
\address{Delio Mugnolo, Institut f\"ur Analysis, Universit\"at Ulm, Helmholtzstra{\ss}e 18, 89081 Ulm, Germany}
\email{delio.mugnolo@uni-ulm.de}

\author{Jean-Fran\c cois Rault}
\address{Jean-Fran\c cois Rault, LMPA Joseph Liouville ULCO, FR CNRS Math. 2956,
Universit\'e Lille Nord de France,
50 rue F. Buisson, B.P. 699, F--62228 Calais Cedex, France}
\email{jfrault@lmpa.univ-littoral.fr}

\date{\today}
\thanks{This article has been initiated during a visit of the first author in Calais and continued during a visit of the second author in Ulm. The first author thanks  the University of Littoral C\^ote d'Opale for its hospitality and financial support. The visit of the second author was  financially supported by the Land Baden--W\"urttemberg in the framework of the \emph{Juniorprofessorenprogramm} -- research project on ``Symmetry methods in quantum graphs''.}

\begin{document}

\begin{abstract}
We are interested in the existence of travelling waves for the Benjamin-Bona-Mahony equation on a network. First we construct an explicit wave, defined in $\mathbb{R}$. Then, we use this wave to derive some conditions on the coefficients appearing in the equations and on the geometry of the network to ensure the existence of travelling waves on the network.
\end{abstract}

\maketitle

\section{Introduction}

Partial differential equations on \emph{networks} (also known as \emph{metric graphs} or \emph{quantum graphs}) have been studied for many years now. While they originate from the pioneering studies of quantum chemists Klaus Ruedenberg and Charles W.\ Scherr in~\cite{RueSch53}, the current state of the art of this theory owns much to the thorough mathematical analysis performed on these problems in the 1980s in particular by G\"unter Lumer and some of his students and collaborators. These and subsequent mathematical investigations have then paved the road to the re-discovery of this kind of models in the context of theoretical quantum physics in the last decade. Perhaps because of the quantum physical bias in the most recent investigations, little attention seems to have been devoted to the study of nonlinear phenomena in networks -- and anyway, mostly from a linear point of view. However, in the theory of nonlinear PDEs one is often interested in questions that are somewhat complementary to those considered in the linear case. For instance, one is often particularly keen on finding out whether a given system modeled by a nonlinear PDE can support \emph{any} travelling wave, for certain conveniently chosen initial data (which then turn out to be suitable inputs), rather than asking for mere existence and uniqueness of a (possibly non-physical) solution for \emph{any} initial data. In the present paper we are going to suggest a possible general approach to the study of travelling waves in networks. It turns out that, unlike on the case of PDEs defined in the whole space (but somewhat similarly to the case of PDEs on Riemannian manifolds~\cite{BerLabHam10}) existence of travelling waves imposes severe restrictions on the geometry of the graph.
The ideas presented here can be applied to a large class of nonlinear equations featuring a second order linear differential operator as the leading order term, including the nonlinear Schr\"odinger equation~\cite{AdaCacFin11}, the one-dimensional Navier-Stokes equation~\cite{PokBor04}, or the FitzHugh--Nagumo  or Rall equations~\cite{EvaKemMaj92,CarMug07}, as soon as a special solution of the PDE without boundary conditions is known. For example, we may as well borrow the travelling wave analysis from~\cite{dePVaz91} in order to discuss  the porous medium reaction-diffusion equation on networks.

\noindent In this note we prefer to fix the ideas and focus primarily on the Benjamin-Bona-Mahony equation
\begin{equation}\tag{BBM}
u_t - au_{xxt} + buu_x +du_x =0 \quad\textrm{ in } \mathbb{R}\textrm{ for } t>0,
\end{equation}
where $a\in(0,\infty)$, $b\in \mathbb{R}^*$ and $d\in \mathbb{R}$. This equation models long waves in  non-linear dispersive systems and is known as a good substitute for the Korteweg-de Vries equation
$$
u_t - u_{xxx} + uu_x =0 \quad\textrm{ in } \mathbb{R}\textrm{ for } t>0
$$
in the case of shallow waters in a channel, see \cite{Whit}. In \cite{BBM1}, Benjamin, Bona and Mahony studied the initial value problem corresponding to (BBM)  and they established global existence and uniqueness results. \\
In \cite{BC}, Bona and Cascaval considered a finite metric tree consisting of edges $\me_i$ and established the well-posedness of a system of BBM equations
\begin{equation}\label{BBM-BC}
u_t - a_iu_{xxt} + b_iuu_x +d_iu_x =0 \quad\textrm{ on each } \me_i \textrm{ for } t>0,
\end{equation}
 with standard continuity conditions and Kirchhoff conditions at vertices $\mv=\me_i \cap \me_j$ and with Dirichlet boundary conditions at endpoints. This equation is used to model the blood flow in the human cardiovascular system, see \cite{BC} and references therein. However, just thinking of the cardiovascular system seems to motivate the discussion of topologies that include circuits.   \\
We begin our note discussing a general way of deriving conditions on the coefficients of (BBM) as well as on the metric and orientation properties of the network (but not on the topological ones!) as a consequence of the standard boundary conditions that are customarily imposed on evolution equations on networks. In Section~\ref{prelim-bbm} we show that this already excludes certain network configurations. This analysis is not really restricted to any specific semilinear PDE of second order (in space).

\noindent
As soon as one focuses on the BBM equation, however, it seems natural to expect that the pressure profile follows a wave-type behavior. Thus, our interest in this note is to derive an explicit formula for the travelling waves solution of (BBM) extending the bifurcation method (Section \ref{Spf}) used by Song and Yang in the case of the Zakharov-Kutnetsov-BBM equations \cite{SY}. Then, in the case where the BBM equation is posed on a network, we use this formula to determine some conditions on the coefficients appearing in the equations and on the geometry of the network to ensure the existence of travelling waves on the network. More precisely, we construct a solitary wave $u$  such that $u(t,x) = \varphi(x-ct)$ on each edge with some $\varphi$ vanishing at $\pm \infty$. For the reader's convenience, we begin with the case where the network is a star (Section \ref{sstar}), then we deal with the case of a tree (Section \ref{stree}), and finally we treat the general case of a network that may possibly contain circuits (Section \ref{scircuits}).

\noindent It seems that not many other investigations on travelling waves or solitons on networks have been performed. In fact, we are only aware of~\cite[Sections 16--17]{BPN} and~\cite{Bel95} for the case of linear diffusion and a certain class of reaction-diffusion equations; and some recent progresses in the theory of linear and nonlinear Schr\"odinger equations on star graphs, see~\cite{AdaCacFin11} and references therein.

\section{Preliminaries on networks}\label{prelim}

For any graph $\Gamma =(V,E,\in)$, the vertex set is denoted by $V=V(\Gamma)$, the edge set by $E=E(\Gamma)$ and the incidence relation by $\in \subset V\times E$. 
All graphs considered in this paper are assumed to be non-empty, simple,  connected and at most countable.
The simplicity property means that $\Gamma$ contains no loops, and at most one edge can join two vertices in $\Gamma$. We give a numbering of the vertices $\mv_i$, $i\in\mathbf V\subset \mathbb{N}$ with ${\rm card}\; {\mathbf V}={\rm card}\;  V$,  and a numbering of the edges $\me_j$, $j\in \mathbf E\subset \mathbb{N}$ with ${\rm card}\; {\mathbf E} ={\rm card}\;  E$. We denote by

\[
N(\mv):=\{j\in \mathbb{N} \ / \ \mv\in \me_j\}.
\]
the of \emph{boundary index set} of $\mv$, i.e., the set of indices of those edges incident in $\mv$; while $d(\mv)$ denotes the \emph{degree} of $\mv$, i.e.,
\[
d(\mv):={\rm card}\; N(\mv)
\]

Also, we assume our graphs to be locally finite, i.e.,
\begin{equation*}\label{locf} 
\forall \ \mv\in V(\Gamma), \ d(\mv) <\infty.
\end{equation*}\\
Recall that a simple, connected graph is called a \emph{tree} if it does not contains any circuits, i.e., if any two vertices are connected by exactly one undirected path; it is called a star if it is a tree and all vertices but one  have degree 1.

We then consider \emph{networks} by associating with a given graph a topological structure  in $\mathbb{R}^m$, i.e., we regard $V(\Gamma)$ as a subset of $\mathbb{R}^m$ (in fact, it is well-known that each countable graph can be embedded in $\mathbb R^3$) and the edge set consists of a collection of Jordan curves
$$E(\Gamma)=\{ \pi_j:[0,\ell_j]\to \mathbb{R}^m \ / \ j\in {\mathbf E}  \}$$
with the following properties: each support
$\me_j:=\pi_j\left([0,\ell_j]\right)$ has its endpoints in the set
$V(\Gamma)$, any two vertices in $V(\Gamma)$ can be connected by a path
with arcs in $E(\Gamma)$, and any two edges $\me_j\not= \me_h$ satisfy
$\me_j\cap \me_h\subset V(\Gamma)$ and $\textrm{card}(\me_j\cap \me_h)\leq 1.$ The arc
length parameter of an edge $\me_j$ is denoted by $x_j$.
Unless otherwise stated, we identify the graph
$\Gamma =(V,E,\in)$ with its associated network
$$G=\bigcup_{j\in {\mathbf E} }\pi_j\left([0,\ell_j]\right),$$
especially each edge $\pi_j$ with its support $\me_j$. $G$ is called
a $\mathcal{C}^\nu$-network, if all $\pi_j\in\mathcal{C}^\nu ([0,\ell_j],\mathbb{R}^m)$.
We shall distinguish the {\it boundary vertices} 
\[
V_b=\{\mv_i\in V \ / \  d(\mv_i)=1\}
\] 
from the {\it ramification vertices}
\[
V_r=\{\mv_i\in V \ / \ d(\mv_i)\geq 2\}.
\]
The orientation the network is given by means of the incidence matrix $\mathcal I:=(\iota_{ij})$, where
\begin{equation}\label{incidence} 
\iota_{ij}=
\begin{cases}
1  & \text{if } \pi_j(\ell_j)=\mv_i,\\
-1  & \text{if } \pi_j(0)=\mv_i,\\
0  & \text{otherwise.} \\
\end{cases}
\end{equation}
The cases of $\iota_{ij}=1$ and $\iota_{ij}=-1$ correspond to the cases of an incoming and an outgoing edge, respectively.\\
The reason why we speak of \emph{networks}, at the risk of confusing a reader more familiar with graph theory, is that our setting is in fact a generalization of the graph theoretical theory of networks, where the edge lengths can be seen as their capacities. Accordingly, we also borrow further graph theoretical notions and speak of a \emph{sink} (resp., a \emph{source}) to describe a vertex with only incoming (resp., only outgoing) edges.

\noindent For a function $u:G \times \mathbb{R} \to \mathbb{R}^+$ we set
$u_j(\cdot,t):=u(\pi_j (\cdot),t):[0,\ell_j]\to\mathbb{R}$ for all $t \geq 0$ and use the abbreviations
$$u_j(t,\mv_i):=u_j(\pi_j^{-1}(\mv_i),t),\quad
\partial_j u_j(t,\mv_i):=\frac{\partial}{\partial x_j}u_j(t,x_j)\Bigr|_{\pi_j^{-1}(\mv_i)}
\quad\hbox{etc.}$$

\section{The BBM equation on networks}\label{prelim-bbm}

Our analysis is based on the classical observation that, by definition, existence of a travelling wave solution for a separable evolution equation
\begin{equation}
\label{generalsepa}
F(t,x,v)=0,
\end{equation}
(where $F$ may possibly depend on partial derivative of any order of $v$), i.e., existence of some function $\psi$ and some constant $c>0$ for which the solution of~\eqref{generalsepa} satisfies
\begin{equation}\label{traveq}
v(t,x)\equiv \psi(x-ct),
\end{equation}
is equivalent to solvability of the system
\begin{equation}
\label{generalsepa2}
\left\{
\begin{array}{rcl}
F(t,x,v)&=&0,\\
v_t&=&-cv_x,
\end{array}
\right.
\end{equation}
because by~\eqref{traveq} the chain rule applied to $\psi$ yield
\[
v_t(t,x)=-c\psi'(x-ct)=-cv_x(t,x),
\]
and conversely~\eqref{traveq} yields the only possible solutions of the second equation of~\eqref{generalsepa2}.
In particular, in the case of the (BBM), this \emph{Ansatz} leads to considering the system
\begin{equation}\label{bbmtransportsystm}
\left\{
\begin{array}{rcl}
u_t-au_{xxt}+buu_x+du_x&=&0\quad\textrm{ in } \mathbb{R}\textrm{ for } t>0,\\
u_t&=&-cu_x \quad\textrm{ in } \mathbb{R}\textrm{ for } t>0.
\end{array}
\right.
\end{equation}
Solving this system amounts to finding a function $w$ of one variable  along with some wave velocity $c\in \mathbb R$ such that
\begin{equation}\label{BBMf}
ac\varphi'''+b\varphi\varphi'+(d-c)\varphi'=0
\end{equation}
where $\varphi'$ denotes the derivative of $\varphi$, with $a>0$, $b\not=0$ and $d\in \mathbb R$. If such a $\varphi$ exists, then by definition $u$ will be obtained by
\[
u(t,x):=\varphi(x-ct),
\]
\emph{on each edge}. Still, existence of a solution of~\eqref{BBMf} is a necessary but not sufficient condition for the whole network to support a travelling wave. In order to construct a travelling wave solution one needs to transform boundary conditions for $u$, which we will introduce soon, into conditions for $\varphi$. We will do so applying an idea developed and thoroughly discussed in~\cite[Sections 16--17]{BPN}.

To fix the ideas, as the basic geometric transition condition at
ramification vertices we impose the continuity condition
\begin{equation}\label{cc} 
\forall \mv_i \in V_r,\ \forall t \in \mathbb{R}^+ :\, \, \me_j\cap \me_h=\{\mv_i\}\, \Longrightarrow \, u_j(t,\mv_i)=u_h(t,\mv_i),
\end{equation}
that clearly is contained in the condition $u\in \mathcal{C}(G\times \mathbb{R})$.
Moreover, at all vertices $\mv_i\in V_r$ we impose the classical Kirchhoff
condition
\begin{equation}\label{kc} 
\forall \mv_i \in V_r,\ \forall t \in \mathbb{R}^+ :\,\, \sum_{j\in {\mathbf E} } \iota_{ij} a_j\partial_j u_j (t,\mv_i)=0,
\end{equation}
where $a_j$ is the coefficient of the BBM~\eqref{BBM-BC} on the edge $\me_j$. This Kirchhoff conditions corresponds to imposing conservation of the flow -- and hence of the mass -- at each ramification vertex. Including the coefficients $a_i$ in this condition is necessary in order to make this conservation independent of the length of the edges (which in turn depend on their parametrization).
Note that Condition \eqref{kc} does not depend on the orientation.

\noindent
Summing up, in the present paper we consider a system of BBM equations on a $\mathcal{C}^2$-network $G$
\begin{equation}\tag{BBMG}
\left\{
\begin{array}{ll}
\partial_t u_i -a_i  \partial_i^2 \partial_t u_i +b_i u_i \partial_i u_i +d_i \partial_i u_i = 0,  \qquad & x\in \me_i,\ t> 0, \\
u_j(t,\mv_p) = u_k(t,\mv_p) &  j,\ k \in N(\mv_p),\ t \geq 0,\\
\displaystyle\sum_{j\in  E } \iota_{pj} a_j\partial_j u_j (t,\mv_p)=0 & t \geq 0,
\end{array}
\right.
\end{equation}
where $a_i>0$, $b_i \in \mathbb{R}^*$, $d_i \in \mathbb{R}$, for all $i \in {\mathbf E} $ and all $\mv_p \in V_r$. \\

\noindent Initial data are not prescribed, since such data would already fix the initial profile of the front wave. Also, we do not impose  boundary conditions  on the boundary vertices $\mv\in V_b$, since  such data describe the tail of the front wave.

\begin{rema}
Another approach would be to replace each edge $\me$ containing a vertex in $V_b$ by a half-line whose endpoint is $\me \cap V_r$. In this way we would consider a kind of nonlinear scattering problem.
\end{rema}

\begin{defi}
A \emph{strong solution} of the system ${\rm(BBMG)}$ is a function
$$u \in \Big\{\  u \in \mathcal{C}(G) \ \Big/ \ \forall \ i \in \mathbf{V}, \ u_i  \in \mathcal{C}^{1,1}( \me_i \times [0,\infty)) \quad\textrm{ and } \ \partial_t u_i  \in \mathcal{C}^{2,0}( \me_i \times  [0,\infty) )\Big\}$$
that satisfies ${\rm(BBMG)}$ pointwise.
\end{defi}
\noindent More specifically, in this paper we are looking for \emph{travelling wave solutions}.
\begin{defi}
A strong solution $u$ of ${\rm(BBMG)}$ is called a \emph{travelling wave} if there exists a velocity vector $(c_i)_{i\in \mathbf E}\subset \mathbb R_+$ and a vector-valued function $(\varphi_i)_{i\in \mathbf E} \subset \mathcal{C}^3(\mathbb{R})$ such that
\[
u_i(x_i,t)=\varphi_i(x_i-c_i t)\qquad\hbox{for all }x_i \in e_i \hbox{ and }t \geq 0.
\] 
\end{defi}

\begin{defi}\label{defi:wave}
A travelling wave $u$ defined by
\[
u_i(x_i,t):=\varphi_i(x_i-c_i t)\qquad\hbox{for all }x_i \in e_i \hbox{ and }t\geq 0,
\]
is said to be \emph{stationary} if on  each edge $e_i$ either  $\varphi_i' \equiv 0$ or $ c_i = 0$. We call $\varphi$ \emph{solitary}  if it admits at most one local extremum and if $\displaystyle \lim_{z \to \pm \infty} \varphi(z) $ exists in $\mathbb{R}$.
\end{defi}

 Observe that the development of a travelling wave in a network on each of whose vertices a boundary condition is imposed requires the existence of paths of infinite length, possibly allowing repetition of edges but not doubling back. Hence, it is apparent that only infinite graphs and/or graphs with circuits can support a travelling wave -- unless we drop any condition on the function at the boundary vertices. It turns out that in fact additional compatibility conditions are necessary.
 
Such conditions can be derived by the standard boundary conditions (continuity conditions \eqref{cc} and Kirchhoff conditions \eqref{kc}) which we impose at the ramification vertices. This idea has been exploited already in~\cite{BPN}. Indeed, if we already assume the solution to be a travelling wave, then the continuity condition~\eqref{cc} is equivalent to
\begin{equation}
\label{eq:cont-eq}
\forall \mv_i \in V_r\, \forall t \in \mathbb{R} ^+:\, \,\me_j \cap \me_k = \{ \mv_i\} \ \Rightarrow \ \varphi_j(\varepsilon_{ij} - c_jt) = \varphi_k(\varepsilon_{ik} - c_kt),
\end{equation}
while derivating both members of~\eqref{cc} with respect to time we see that
\begin{equation}
\label{eq:kirchh-eq}
\forall \mv_i \in V_r\, \forall t \in \mathbb{R}^+ :\, \,\me_j \cap \me_k = \{ \mv_i\} \ \Rightarrow \ c_j \varphi_j'(\varepsilon_{ij} - c_jt) =c_k \varphi_k'(\varepsilon_{ik} - c_kt) \ ,
\end{equation}
with 
\[
\varepsilon_{ij}: =\frac{ \ell_j (1+\iota_{ij})}{2}=
\begin{cases}
\ell_j  & \text{if } \pi_j(\ell_j)=\mv_i,\\
0 & \text{if } \pi_j(0)=\mv_i,
\end{cases}
\qquad \hbox{for all }i,j\hbox{ s.t. }\iota_{ij}\not=0.
\]
\begin{rema}
Using the variable $z=x-c_jt$ we see that~\eqref{eq:cont-eq} can be equivalently re-written as
\begin{equation}\label{prof1}
\forall \mv_i \in V_r\, :\, \,\me_j \cap \me_k = \{ \mv_i\} \ \Rightarrow\varphi_j(z) = \varphi_k\Big(\varepsilon_{ik}-\frac{c_k}{c_j}\varepsilon_{ij}+\frac{c_k}{c_j}z \Big).
\end{equation}
Hence, the continuity at ramification vertices implies that for any fixed pair of mutually adjacent edges $\me_j,\me_k$, each $\varphi_j$ is of the form
\begin{equation}\label{profuni}
\varphi_j(z) = \varphi_k\Big(C(k,j)+\frac{c_k}{c_j}z \Big),
\end{equation}
for a constant $C(k,j)$. But then, owing to connectedness of the graph,~\eqref{profuni} can be extended to any pair of edges $\me_j,\me_k$, 
where $C(k,j)$ is a constant depending on the edge lengths and speeds along a suitable path containing $\me_j$ and $\me_k$. This shows that a travelling wave is completely determined by its profile on one single edge $\me_k$ and by the speeds $c_1, \ c_2,\ldots$
Using \eqref{prof1}, we finally  observe that
$$
\me_j \cap \me_k = \{ \mv_i\} \ \Rightarrow \  \partial_j u_j(t,\mv_i) = \frac{c_k}{c_j} \varphi_k'(\varepsilon_{ik} - c_k t),\qquad t \geq 0.
$$
Combining this relation with~\eqref{eq:kirchh-eq} we see that a travelling wave $u$ satisfies the Kirchhoff condition \eqref{kc} at a ramification vertex $\mv_i$ if and only if
\begin{equation}\label{Kirc}
\sum_{j\in {\mathbf E} } \iota_{ij}  \frac{a_j}{c_j}=0.
\end{equation}
We remark that, unlike~\eqref{eq:kirchh-eq} and~\eqref{eq:cont-eq}, the compatibility condition~\eqref{Kirc} does not impose any restriction on the geometry of the network, but only on the coefficients appearing in (BBM) on adjacent edges. We also remark that the system of ODEs consisting of the equation~\eqref{BBMf} on each edge is underdetermined, since the leading term $\varphi'''$  is of third order but we are only imposing conditions~\eqref{eq:cont-eq} and~\eqref{eq:kirchh-eq} on $\varphi$ and $\varphi'$.
\end{rema}

\begin{lem}
If a travelling wave solution of ${\rm(BBMG)}$ is stationary, then it is constant.
\end{lem}

\begin{proof} By Definition~\ref{defi:wave}, the travelling wave $u$ is stationary if and only if on each edge $\me_j$ either $\varphi_j$ is a constant function or $c_j$ vanishes. In either case, $\varphi_j$ will be constant in time (see~\eqref{BBMf}) on each edge $\me_j$. In view of the continuity conditions \eqref{cc} and because $G$ is connected by assumption, the solution $u$ will be constant in time, as well. 
\end{proof}
Thus, we focus on the case where the wave $u$ really travels, i.e.  $u$ is non-constant and all the $c_i$ are strictly positive. 

\begin{lem}
If a non-constant travelling wave solution of ${\rm(BBMG)}$ exists, then no ramification vertex can be either a sink or a source.
\end{lem}
\begin{proof}
Since we restrict ourselves to the case $a_j>0$ and $c_j>0$, if a non-constant travelling wave exists in $G$, then it follows from~\eqref{Kirc} that neither 
\[\displaystyle \sum_{j\in {\mathbf E} } \iota_{ij} \frac{a_j}{c_j} >0
\]
(as it is the case, in particular, for a sink -- i.e., if $\iota_{ij}>0$ for all $j \in{\mathbf E}$) nor 
\[
\displaystyle \sum_{j\in {\mathbf E} } \iota_{ij} \frac{a_j}{c_j}<0
\] (as, in particular, for a source -- i.e., if $\iota_{ij}<0$ for all $j \in{\mathbf E}$ ) can occur at any vertex $\mv_i \in V_r$.
\end{proof}

We have actually proved slightly more: If in fact $a_j=c_j$ for all $j\in \mathbf E$, then the above lemma states that each vertex has to be \emph{balanced}, in the sense that each vertex has to have the same number of outgoing and incoming edges. This condition also appears in the theory of first order differential operators on network~\cite{Exn12}, and is known to be necessary for the existence of travelling wave solutions for a nonlinear Schr\"odinger equation on a star graph~\cite{AdaCacFin11}. Of course, in classical graph theory it is well-known that a directed graph is Eulerian if and only if it is balanced.

\begin{rema}
Just as continuity and Kirchhoff-type conditions are natural for a manifold of partial differential equations on networks, and not only for the $\rm(BBMG)$ system, the compatibility conditions~\eqref{eq:cont-eq} and~\eqref{Kirc} are the natural one for all problems on networks involving differential operators in divergence form. Results similar to those of this note could then be deduced for different classes of evolution equations.

In fact, even if we have chosen to concentrate on the BBM equation, the above analysis shows how our strategy can be applied to more general problems, as those mentioned in the introduction. Though, other favorite differential models for waveguides -- like the KdV equation, the Camassa--Holm or the Whitham equation -- do not seem to be treatable by our methods, as it is not quite clear which boundary conditions should be imposed on a third order differential equations or on integral equations~\cite{CamHol93,AdlGurGur99,EhrKal09}. 
\end{rema}

\section{Profile of the front}\label{Spf}

We want to determine an explicit function $\varphi:\mathbb{R}\to \mathbb R$ such that $u$ defined by
\[
u(t,x): = \varphi(x-ct),\qquad (t,x)\in \mathbb R \times \mathbb R^+,
\]
for some $c>0$, is a solution of (BBM). Integrating~\eqref{BBMf} over $x$, we are lead to
\begin{equation*}
-c\varphi + ac\varphi'' + \frac{b}{2}\varphi^2 +d\varphi =A,
\end{equation*}
for some $A\in \mathbb{R}$. With the notation $\psi:=\varphi'$, we obtain the first order differential system
\begin{equation}\label{BBMs}
\left\{ 
\begin{array}{lll}
\varphi' &=& \psi, \\
\psi'&=& \frac{A +(c-d)\varphi - b\varphi^2/2}{ac}.
\end{array}
 \right. 
\end{equation}
As in \cite{SY}, we study the phase portraits of this system. First, we consider the functional
\begin{equation}\label{Func}
H(\varphi, \psi) = \frac{\psi^2}{2} - \frac{1}{ac} \Bigg( A\varphi +\frac{c-d}{2}\varphi^2 -\frac{b}{6}\varphi^3  \Bigg).
\end{equation}
A direct calculation shows that $H$ is constant along any trajectory of \eqref{BBMs}, i.e. for all  solutions $y\mapsto (\varphi(y), \psi(y) )$ of system~\eqref{BBMs} we have
$$
\frac{d}{dy} H(\varphi(y), \psi(y) ) = 0.
$$
Then, we need to investigate the stationary points of \eqref{BBMs}. 

 \begin{itemize}
\item If 
\[
(c-d)^2 +2Ab \leq 0,
\]
then \eqref{BBMs} has at most one stationary point and all its trajectories are unbounded, see Figure \ref{pp}.
\end{itemize}

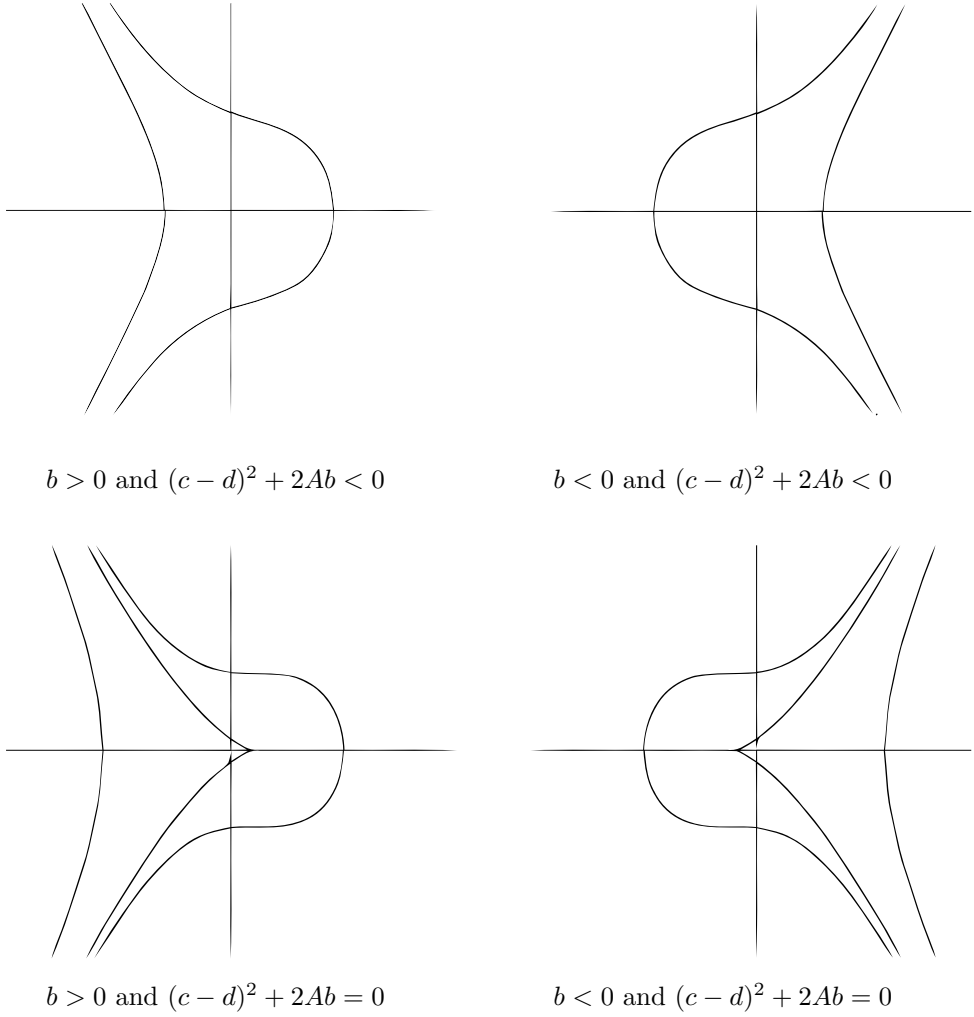
\begin{figure}[h]
\begin{picture}(340,375)
\put(0,10){\includegraphics[width=1\textwidth]{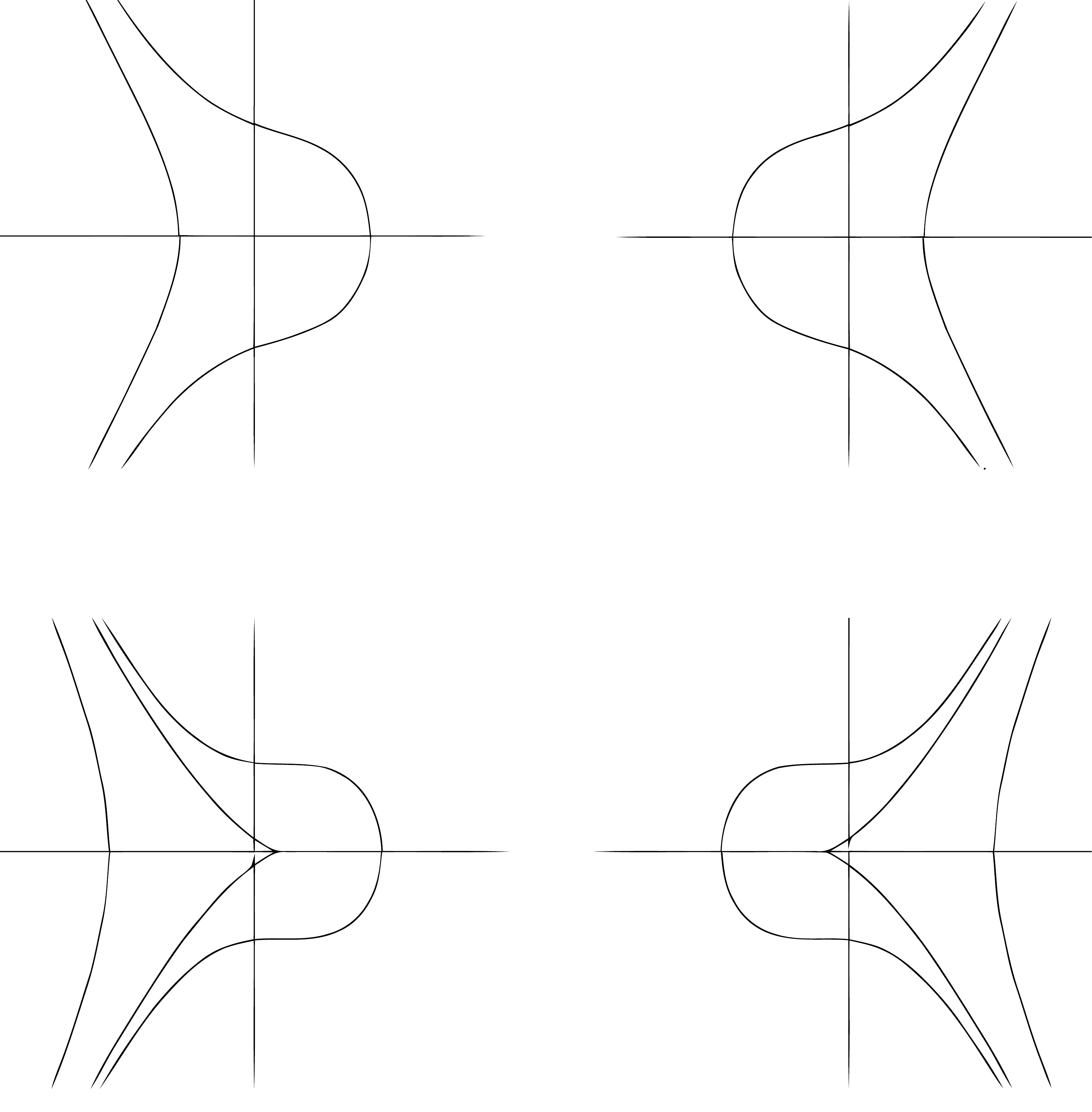}}
\put(15,195){$b>0$ and $(c-d)^2+2Ab<0$}\put(205,195){$b<0$ and $(c-d)^2+2Ab<0$}
\put(15,0){$b>0$ and $(c-d)^2+2Ab=0$}\put(205,0){$b<0$ and $(c-d)^2+2Ab=0$}
\end{picture}
\caption{Phase portraits of \eqref{BBMs} with $(c-d)^2 +2Ab\leq 0 $.}\label{pp}
\end{figure}

Hence, we focus on the case $(c-d)^2 +2Ab > 0$.

\begin{itemize}
 \item If 
 \[ (c-d)^2 +2Ab>0\]
  then \eqref{BBMs} admits two stationary points
$$
p_1 = \Big(\frac{c-d-\sqrt{(c-d)^2 +2Ab}}{b},0 \Big),
$$ 
and 
$$
p_2 =\Big( \frac{c-d+\sqrt{(c-d)^2 +2Ab}}{b},0 \Big).
$$
Clearly, the eigenvalues $\lambda$ of the linearized system of \eqref{BBMs} around $p_1$ satisfy
$$
\lambda^2 = \frac{\sqrt{(c-d)^2 +2Ab}}{ac}
$$
and the eigenvalues $\mu$ of the linearized system of \eqref{BBMs} around $p_1$ satisfy
$$
\mu^2 = -\frac{\sqrt{(c-d)^2 +2Ab}}{ac}.
$$
Then according to the theory of dynamical systems (e.g.  \cite{Am}, \cite{Chow} and \cite{Guck}), we obtain that $p_1$ is a saddle point for \eqref{BBMs}, whereas $p_2$ is a center.\\

\begin{figure}[h]
\begin{picture}(340,134)
\put(0,0){\includegraphics[width=1\textwidth]{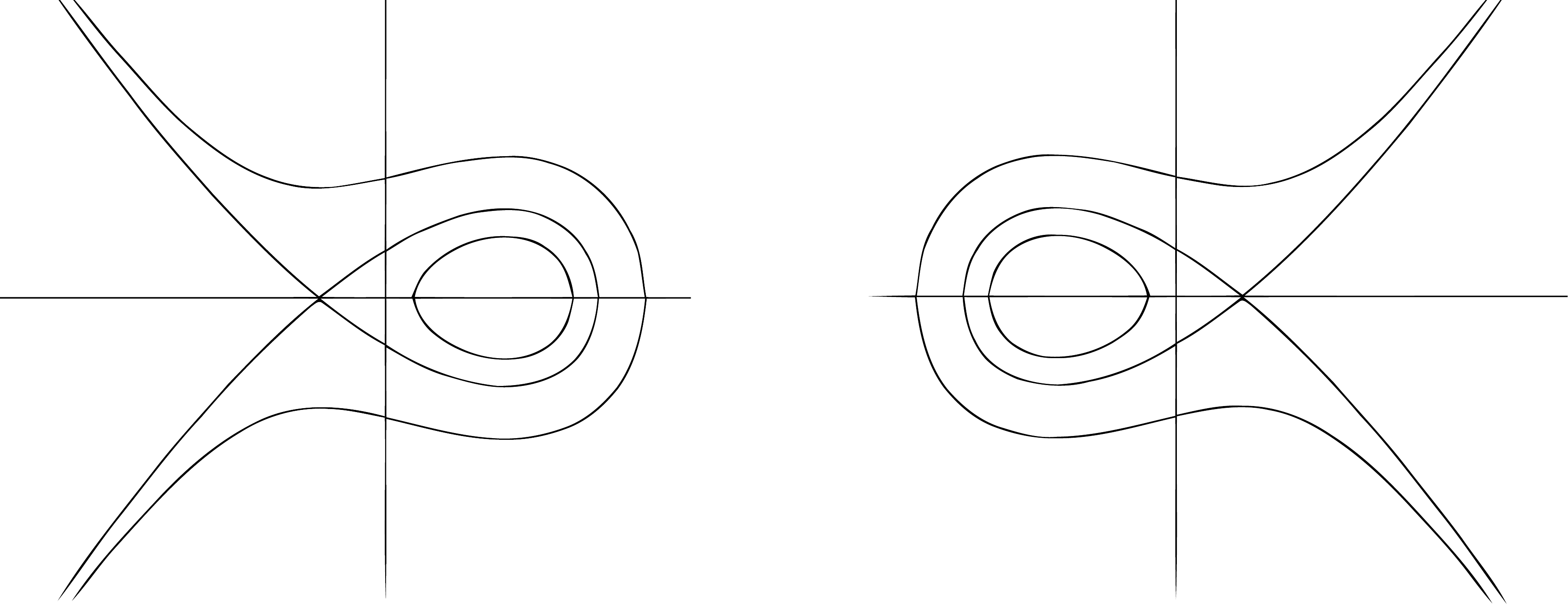}}
\put(10,20){$\Sigma_2$}\put(20,100){$\Sigma_1$}
\put(340,20){$\Sigma_2$}\put(325,100){$\Sigma_1$}
\put(78,83){$\Gamma$}\put(280,83){$\Gamma$}
\put(79,96){$>$}\put(260,98){$>$}
\put(112,83){$>$}\put(239,83){$>$}
\put(113,49){$<$}\put(239,49){$<$}
\put(37,103){\Large{$ \lrcorner$}}\put(37,35){\Large{$ \llcorner$}}\put(297,93){\Large{$ \urcorner$}}\put(297,36){\Large{$ \ulcorner$}}
\put(69,61){$p_1$}\put(282,61){$p_1$} \put(71,69){$\bullet$}\put(284,69){$\bullet$}
\put(110,62){$p_2$}\put(240,62){$p_2$} \put(114,69){$\bullet$}\put(239,69){$\bullet$}
\end{picture}
\caption{Phase portraits of \eqref{BBMs} with $(c-d)^2 +2Ab >0 $.}\label{pp2}
\end{figure}
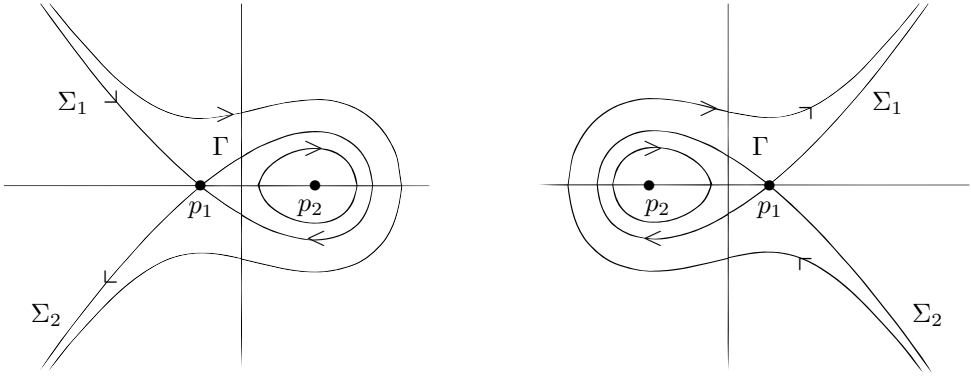

\item Again in the case 
\[(c-d)^2 +2Ab>0,\]
we are specially interested in the homoclinic orbit $\Gamma$.  The heteroclinic ones also represent  travelling waves, but we can not derive their explicit formula (see Remark \ref{perOrb}), and some of them are singular at $\pm \infty$ (e.g. the branches $\Sigma_1$ and $\Sigma_2$.). Indeed, a homoclinic orbit corresponds to a solution $(\varphi, \psi)$ of the \eqref{BBMs} defined on $\mathbb{R}$ and satisfying
\begin{equation}\label{holim}
\lim_{y \to \pm \infty} \varphi(y) = \frac{c-d-\sqrt{(c-d)^2 +2Ab}}{b} \ \quad\textrm{ and } \ \lim_{y \to \pm \infty} \psi(y) = 0.
\end{equation}
In the special case $A=0$,~\eqref{holim} reduces to
\begin{equation}\label{holim0}
\lim_{y \to \pm \infty} \varphi(y) =  \lim_{y \to \pm \infty} \psi(y) = 0,
\end{equation}
and recalling that $H$ is constant along any trajectory of  \eqref{BBMs}, we obtain $H(\varphi, \psi)=0$ along $\Gamma$, and in the $\varphi-\psi$ plane, using \eqref{Func}, $\Gamma$ can be described as 
\begin{equation}\label{ppdes}
\psi^2 =  \frac{c-d}{ac}\varphi^2 -\frac{b}{3ac}\varphi^3 .
\end{equation}
Up to a translation, we can suppose that $(\varphi(0), \psi(0)) =(\varphi_0,0)$. Then, from standard regularity results (see~\cite{Am}) the abscissa $\varphi$ of trajectory $\Gamma$ is a solution of~\eqref{BBMf} belonging to $\mathcal{C}^\infty(\mathbb{R})$. When $b>0$, $\varphi $ is positive in $\mathbb{R}$, increasing in $(-\infty,0)$ because $\psi(t) = \varphi'(t)$ lies in upper half-plane for $t<0$, and decreasing in $(0,\infty)$ because $\psi$ lies in lower half-plane when $t>0$. When $b<0$, $\varphi$ is negative in $\mathbb{R}$, decreasing in $(-\infty,0)$ and increasing in $(0,\infty)$. In both cases, it is easy to see that $\varphi$ is a solitary wave of~\eqref{BBMf}, see Figure \ref{wave}.
 \end{itemize} 

\begin{figure}[h]
\begin{center}
\includegraphics[width=0.8\textwidth]{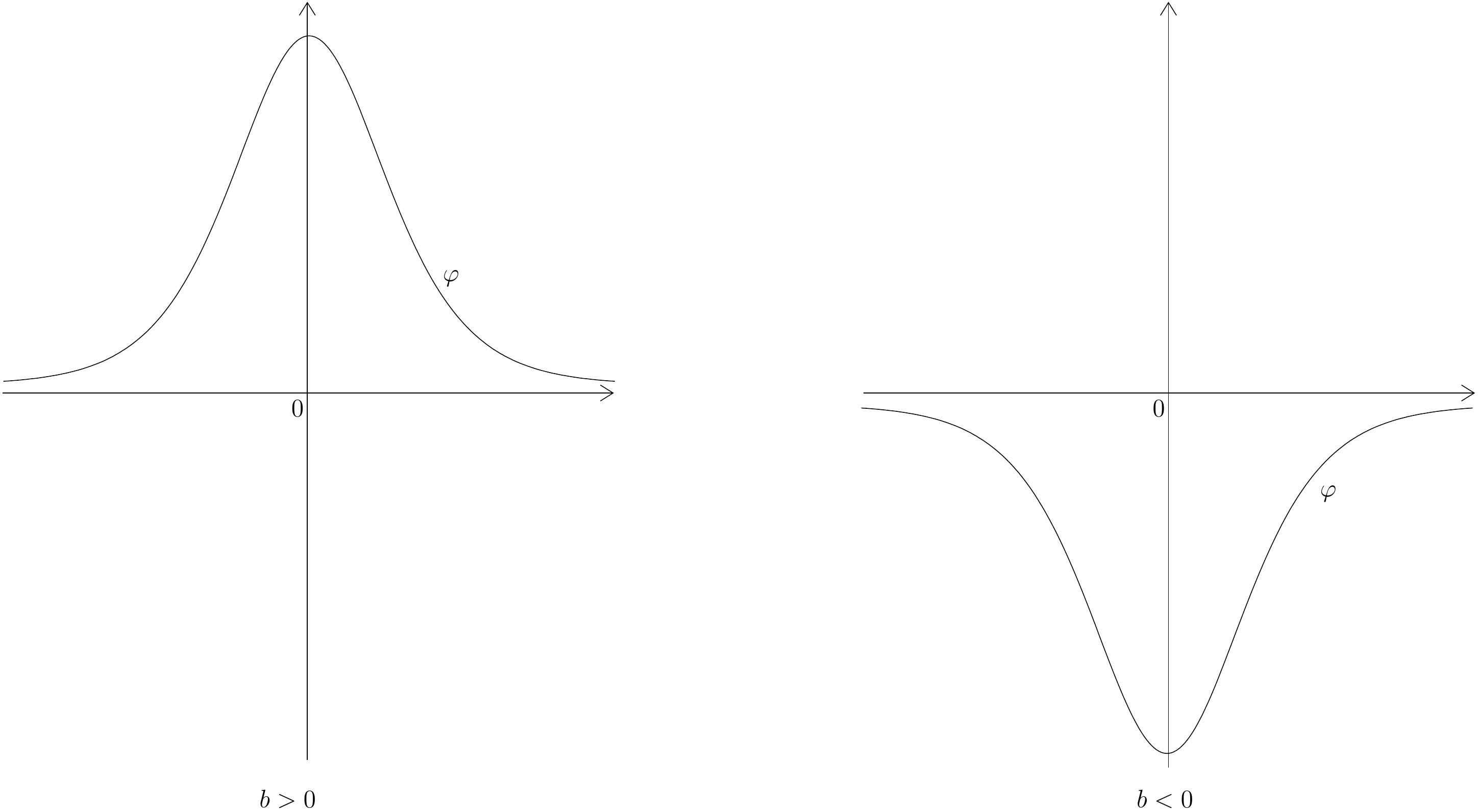}
\caption{Solitary wave of~\eqref{BBMf}.}\label{wave}
\end{center}
\end{figure}

\noindent From~\eqref{ppdes} and writing $\psi = \varphi'= \frac{d \varphi}{d y}$, we obtain
\begin{equation*}
\pm \sqrt{ \frac{1}{ \frac{c-d}{ac}\varphi^2 -\frac{b}{3ac}\varphi^3 }} d \varphi = d y,
\end{equation*}
and integrating between $0$ and $y$, we have
\begin{equation*}
\mathrm{sgn}(b) \sqrt{\frac{ac}{c-d}}\int_{\varphi(y)}^{\varphi_0}\frac{1}{\sqrt{ s^2 -\frac{b}{3(c-d)}s^3 }} d s = \int_y^0 d y = -y,\qquad \quad\textrm{ for } y<0,
\end{equation*}
and 
\begin{equation*}
-\mathrm{sgn}(b) \sqrt{\frac{ac}{c-d}}\int_{\varphi_0}^{\varphi(y)}\frac{1}{\sqrt{s^2 -\frac{b}{3(c-d)}s^3 }} d s = \int_0^y d y = y,\qquad \quad\textrm{ for } y>0.
\end{equation*}
In order to compute explicit formulas, we need to impose
\begin{equation}\label{c>d}
c>d.
\end{equation}
Using both anti-derivatives $x\mapsto \textrm{arccosh} \Big( \frac{\alpha}{x} -1 \Big)$ and $x\mapsto  \textrm{arccosh} \Big( \frac{\alpha}{x} +1 \Big)$ with $\alpha = \pm\frac{6(c-d)}{b}$, and up to a translation, we obtain:

\begin{theo}\label{tsw}
Let $A=0$. When $a>0$, $b\not=0$, $d\in \mathbb{R}$ and $c>\max\{d,0\}$,~\eqref{BBMf} admits the solitary waves
\begin{equation}\label{smsw}
\varphi(y) = \frac{6(c-d)}{b} \cdot \frac{1}{1+ \cosh \Big( \sqrt{\frac{c-d}{ac}} y \Big)}\qquad \quad\textrm{ for } y \in \mathbb{R}
\end{equation}
and
\begin{equation}\label{sisw}
\varphi(y) = \frac{6(c-d)}{b} \cdot \frac{1}{1- \cosh \Big( \sqrt{\frac{c-d}{ac}} y \Big)},\qquad \quad\textrm{ for } y \in \mathbb{R^*}.
\end{equation}
\end{theo}

\begin{rema}
The smooth solitary wave $\varphi$ describes the homoclinic orbit $\Gamma$, and the singular solitary wave $\varphi$ describes the branches $\Sigma_1$ and $\Sigma_2$.
\end{rema}

\noindent If we do not impose $A=0$ in~\eqref{holim0}, we obtain a solitary $\varphi_A$ wave solution of~\eqref{BBMf} satisfying \eqref{holim}. Let $\varphi_0 =  \varphi_A - A$. From, \eqref{BBMf}, we obtain that
\begin{equation}\label{trans}
-c\varphi_0' + ac\varphi_0''' + b\varphi_0 \varphi_0' +(d+A)\varphi_0' =0,
\end{equation}
and $\varphi_0$ verifies \eqref{holim0}. Hence, from Theorem \ref{tsw}, we have:

\begin{coro}
Let $A\in \mathbb R$. If $a>0$, $b\not=0$, $d\in \mathbb{R}$ and $c>\max\{d+A,0\}$ satisfy $(c-d-A)^2 + 4Ab>0$,~\eqref{BBMf} admits the the solitary wave solutions
\begin{equation*}
\varphi_A(y) = A+ \frac{6(c-d+A)}{b} \cdot \frac{1}{1+ \cosh \Big( \sqrt{\frac{c-d+A}{ac}} y \Big)},\qquad \quad\textrm{ for } y \in \mathbb{R}
\end{equation*}
and
\begin{equation*}
\varphi_A(y) =A+ \frac{6(c-d+A)}{b} \cdot \frac{1}{1- \cosh \Big( \sqrt{\frac{c-d+A}{ac}} y \Big)},\qquad \quad\textrm{ for } y \in \mathbb{R^*}.
\end{equation*}
\end{coro}

\begin{rema}\label{perOrb}
On can see in Figure  \ref{pp2} that there also exist some periodic orbit of \eqref{BBMs}. It corresponds to a periodic wave for \eqref{BBMf}, but we are unable to derive explicit formulas, even if it is possible to express it in terms of the Weierstra\ss 's  function, see \cite{EKNN,WW}.
\end{rema}

\section{BBM equation on a  star}\label{sstar}

Let us consider a semi-infinite star, i.e., the finite graph with one vertex $\mv_1$  and $N$ edges of semi-infinite length. In view of \eqref{Kirc}, we suppose that the incidence vector $(\iota_{1i})_{1 \leq i \leq N}$, defined in \eqref{incidence}, is not $\pm (1)_{1 \leq i \leq N}$. Thus, up to relabeling, there exists $1\leq L<N$ such that 
\[
\iota_{1j} =\left\{
\begin{array}{ll}
1\qquad &\hbox{ for }1 \leq j \leq L,\\
-1 &\hbox{for  }L+1\le j \leq N.
\end{array}
\right.
\]
In other words, we have $L$ incoming edges $\me_1,\ldots,\me_L$ and $N-L$ outgoing edges $\me_{L+1},\ldots,\me_{N}$.
We identify the incoming edge $\me_j$ with the half-line $(-\infty_j,0]$, and the outgoing edges $\me_k$ with the half-line $[0,\infty_k)$.\\

\begin{figure}[h]
\begin{center}
\includegraphics[width=0.5\textwidth]{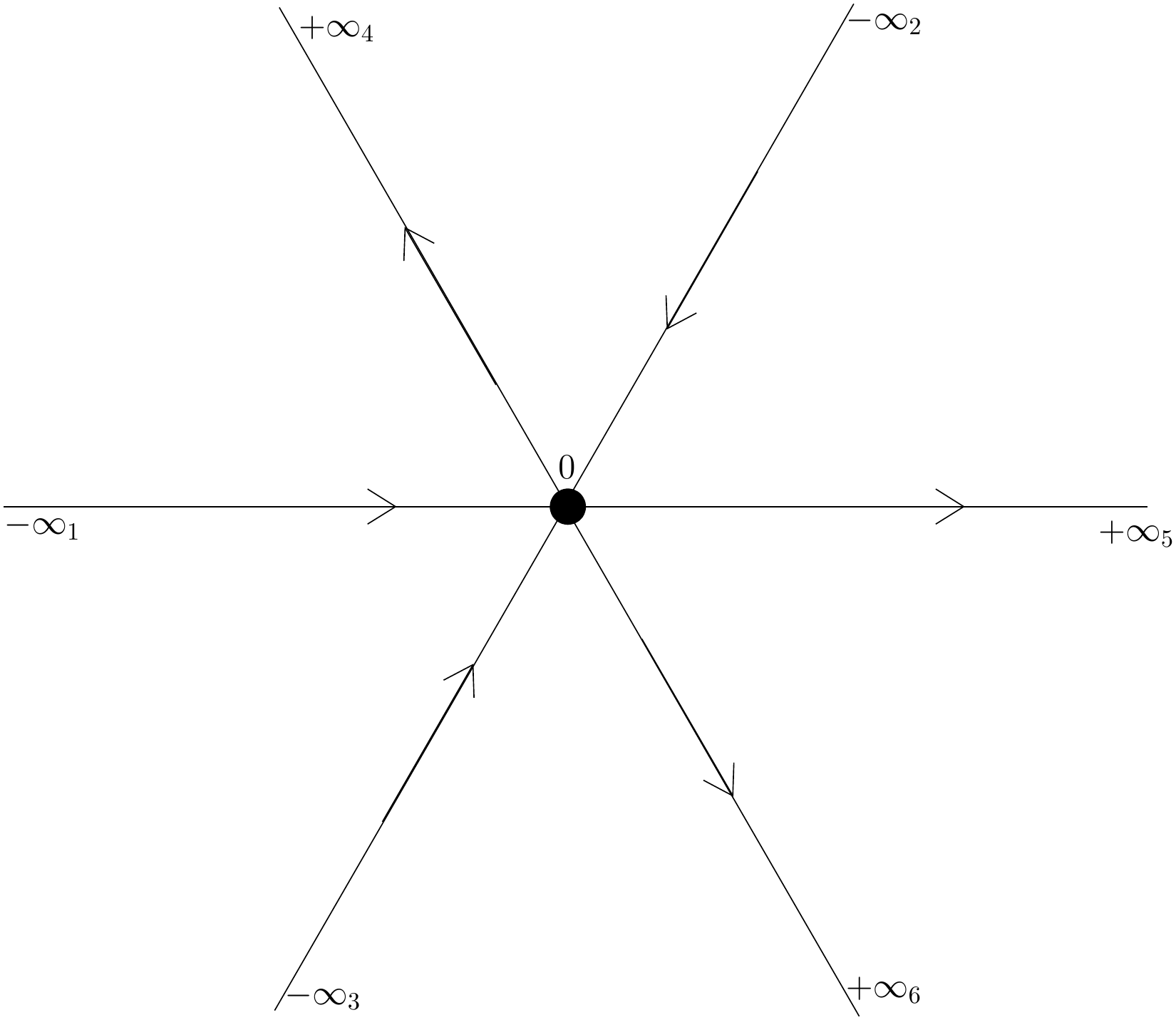}
\caption{A star with infinite edges.}\label{star}
\end{center}
\end{figure}

\noindent We want to construct a solution $u$ in the form $u_i(t,x)= \varphi_i(x-c_it)$ on each edge $\me_i$, where $\varphi_i$ is defined in accordance with~\eqref{smsw} by
\begin{equation}\label{smsw-s}
\varphi_i(z) := \frac{6(c_i-d_i)}{b_i} \cdot \frac{1}{1+ \cosh \Big( \sqrt{\frac{c_i-d_i}{a_ic_i}} z \Big)},\qquad  z \in \mathbb{R}.
\end{equation}
As in \eqref{prof1}, the contintuity condition \eqref{cc} at $\mv_1 =0$ leads to 
\begin{equation}\label{profS}
\varphi_1(z)=\varphi_i \Big(\frac{c_i}{c_1}z\Big), \quad\textrm{ for all } z\in\mathbb{R} \quad\textrm{ and } 1 \leq i \leq N.
\end{equation}
Hence,~\eqref{smsw-s} and \eqref{profS} imply that the continuity condition  \eqref{cc} is satisfied if
\begin{equation}\label{ampl}
\frac{c_1-d_1}{b_1} =\frac{c_i-d_i}{b_i}  \quad\textrm{ for all } 1 \leq i \leq N,
\end{equation}
and
\begin{equation}\label{larg}
c_1 \sqrt{ \frac{c_1-d_1}{a_1 c_1} } =c_i \sqrt{\frac{c_i-d_i}{a_i c_i}}  \quad\textrm{ for all } 1 \leq i \leq N.
\end{equation}
Then,~\eqref{Kirc} implies that the Kirchhoff condition \eqref{kc} is satisfied if
\begin{equation}\label{kc-s}
\sum_{i=1}^L \frac{a_i}{c_i} = \sum_{j=L+1}^N \frac{a_j}{c_j}.
\end{equation}
Similarly as in~\eqref{trans}, $v$ is solution of ${\rm(BBMG)}$ with $d_i \not = 0$ if and only if 
\[
u_{|_{\me_i}}:=\mv_{|_{\me_i}}+\frac{d_i}{b_i}
\] 
is a solution of the modified system
\begin{equation}\tag{BBMG'}
\left\{
\begin{array}{ll}
\partial_t u_i -a_i  \partial_i^2 \partial_t u_i +b_i u_i \partial_i u_i = 0 & \quad\textrm{ on each } \me_i \textrm{ for } t> 0, \\
u_j(v,t) = u_k(v,t) & \quad\textrm{ for } t\geq 0,\  \forall \ j,\ k \in N(\mv), \\
\displaystyle\sum_{j\in {\mathbf E} } \iota_{ij} a_j\partial_j u_j (v,t)=0 & \quad\textrm{ for } t\geq 0,\ \forall \ j \in N(\mv), 
\end{array}
\right.
\end{equation}

\noindent Thus, without any loss of generality, we can suppose 
\begin{equation}\label{simpl}
d_i=0 \ \quad\textrm{ for all } \ 1\leq i\leq n.
\end{equation}
Summing up we have obtained.

\begin{theo}\label{thm_star}
Let~\eqref{simpl} hold. If the coefficients $a_i >0$ and $b_i \in \mathbb{R}^*$ satisfy the compatibility conditions
\begin{equation}\label{ccc}
\sqrt{\frac{a_i}{a_1} }= \frac{b_i}{b_1} >0 \ \quad\textrm{ for all } \ 1\leq i\leq n
\end{equation}
and
\begin{equation}\label{cck}
\sum_{i=1}^L b_i = \sum_{j=L+1}^N b_j,
\end{equation}
then there exists a solution $u$ of ${\rm(BBMG)}$ of the form
\[
u_{|_{\me_i}}(t,x) = \varphi(x_i-c_it + \tau_i),
\]
where $\varphi$ is defined as in~\eqref{smsw} with
\begin{equation}\label{ps}
c_1 >0 \ \quad\textrm{ and } \ c_i = \sqrt{ \frac{a_i}{a_1} } c_1.
\end{equation}
\end{theo}
\begin{proof}
Combining~\eqref{ccc} and \eqref{ps}, we obtain~\eqref{ampl} and \eqref{larg} with $d_i=0$ for all $1\leq i \leq N$. Then, using the definition of the propagation speeds and \eqref{ccc}, we have
$$
\frac{a_i}{c_i} = \frac{a_i \sqrt{a_1}}{\sqrt{a_i}c_1} = \frac{\sqrt{a_1}}{c_1} \sqrt{a_i} = \frac{\sqrt{a_1}}{c_1}  \Big( \frac{b_i}{b_1}  \sqrt{a_1} \Big).
$$
Thus, up the positive constant $\frac{a_1}{b_1c_1}$,~\eqref{cck} is equivalent to \eqref{kc-s}.
\end{proof}

\noindent In the case $N=3$, recalling the results by Bona and Cascaval in~\cite{BC}, when the initial data $u_0$ is the initial profile of a wave
$$
{u_0}_{|_{{\me_i}}}(x_i) = \varphi_i(x_i), \quad\textrm{ for all } x_i \in \me_i, \quad\textrm{ for all } 1\leq i \leq N,
$$
then the unique solution of ${\rm(BBMG)}$ is the solitary wave built in Theorem \ref{thm_star}. Moreover, our compatibility conditions \eqref{ccc} and \eqref{cck} do not seem to be artificial in view of the numerical computations in \cite{BC} which exhibit a reflected wave in the case $N=3$ and $a_i=b_i=d_i=1$ for all $i$.

\begin{rema}
In view of~\eqref{ampl} and \eqref{larg}, if all the coefficients $a_i$, $b_i$ and $d_i$ are equal, i.e., if they all agree with some common value $a$, $b$ and $d$, then all the propagation speeds are equal. Thus the Kirchhoff condition \eqref{kc-s} is satisfied if and only if 
\[
N-L=L.
\]
\end{rema}

\section{BBM equation on a tree}\label{stree}
Now, we consider the case where the graph is a directed tree without boundary conditions at boundary vertices. We do not regard our tree as rooted, and in particular at each edge there may be more than one incoming edge.\\

\begin{figure}[h!]
\begin{center}
\includegraphics[width=0.6\textwidth]{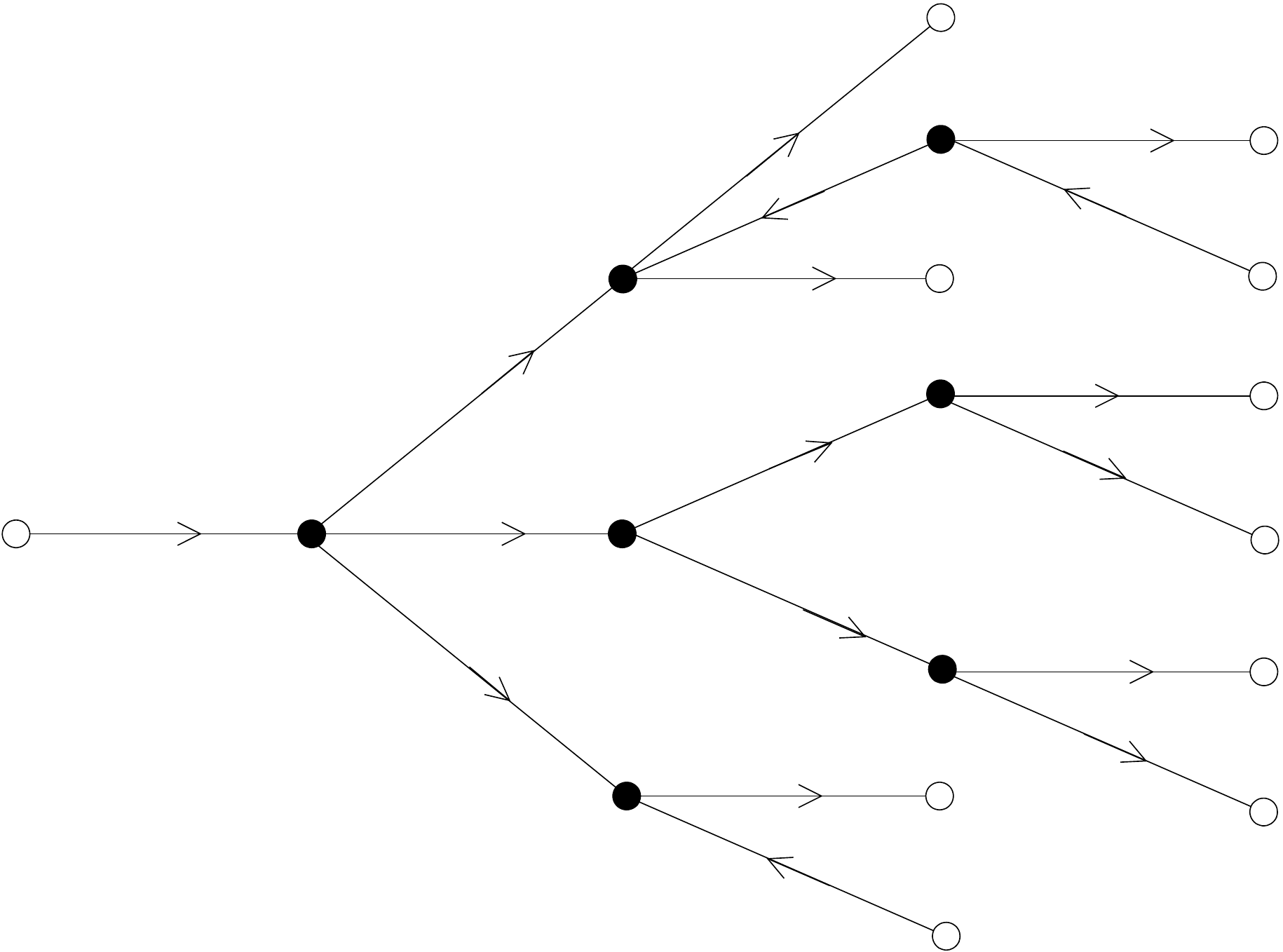}
\caption{A finite tree.}\label{tree}
\end{center}
\end{figure}

\noindent As in the previous section, we want to construct a solution which is a solitary wave. Since paths with more than two edges can occur, we need to add a parameter $\tau_i$ in $u_{|_{\me_i}}$ to account for the edge lengths. Thus, we look for a solution  $u$ in the form 
\[
u_i(t,x)= \varphi_i(x-c_it + \tau_i),\qquad i=1,\ldots,n,
\]
where $\varphi_i$ is defined by \eqref{smsw}.\\

\noindent From the continuity condition \eqref{cc}, we deduce that for any vertex $\mv_k$ and any edges $\me_i$ and $\me_j$ such that $\iota_{ki}=1$ and $\iota_{kj}=-1$, we have
$$
 \varphi_i(\ell_i-c_it +\tau_i) =\varphi_j(-c_jt +\tau_j).
$$
Hence, adjusting~\eqref{ampl} and\eqref{larg}, we need to satisfy the following conditions
\begin{equation}\label{ampl-t}
\frac{c_j-d_j}{b_j} =\frac{c_i-d_i}{b_i}  \quad\textrm{ for all } i,j \in N(\mv_k), \quad\textrm{ for all }\mv_k\in V_r,
\end{equation}
\begin{equation}\label{larg-t}
c_j\sqrt{ \frac{c_j-d_j}{a_j c_j} } = c_i \sqrt{\frac{c_i-d_i}{a_i c_i}}  \quad\textrm{ for all } i,j \in N(\mv_k), \quad\textrm{ for all }\mv_k\in V_r
\end{equation}
and 
\begin{equation}\label{transl-t}
\sqrt{ \frac{c_j-d_j}{a_j c_j} } (\ell_j +\tau_j) = \sqrt{\frac{c_i-d_i}{a_i c_i}} \tau_i \quad\textrm{ for all } i,j \in N(\mv_k), \quad\textrm{ for all }\mv_k\in V_r
\end{equation}
Up to relabelling we can of course assume that $\mv_0$ is the root of the tree. Moreover,  we can choose $\tau_0=0$ assuming  $\mv_0 \in \me_0$. Since our graph is a tree, each vertex $\mv_i$ is linked to $\mv_0$ by exactly one path. Let us denote it by $(\me_{i_1}, \ \me_{i_2}, \ \cdots, \ \me_{i_k})$ for some $k \in \mathbb{N}$. Thus, using~\eqref{transl-t}, $\tau_i$ can be uniquely determined by the lengths $( \ell_{i_1},\ \cdots, \ \ell_{i_k})$ and by the coefficients $a_{i_j}, \ b_{i_j}, \ c_{i_j}$ and $d_{i_j}$ appearing along the path from $\mv_0$ to $\mv_i$. As in the previous section, we can without loss of generality assume that~\eqref{simpl} holds. Following the idea of Theorem \ref{thm_star} we obtain:

\begin{theo}\label{thm_tree}
Let~\eqref{simpl} hold. If the coefficients $a_i >0$ and $b_i \in \mathbb{R}^*$ satisfy the compatibility conditions
\begin{equation}\label{ccc-t}
\sqrt{\frac{a_i}{a_j} }= \frac{b_i}{b_j} >0  \quad\textrm{ for all } i,j \in N(\mv_k), \quad\textrm{ for all }\mv_k\in V_r
\end{equation}
and
\begin{equation}\label{cck-t}
\sum_{i  \in  \mathbf{E}} \iota_{ki} b_i = 0,  \quad\textrm{ for all }\mv_k\in V_r
\end{equation}
then there exists a solution $u$ of ${\rm(BBMG)}$ of the form
\[
u_{|_{\me_i}}(t,x) = \varphi(x_i-c_it + \tau_i)
\]
where $\varphi$ is defined by~\eqref{smsw} and the propagation speeds are given by
\begin{equation}\label{ps-t}
c_0 >0 \ \quad\textrm{ and } \ c_i = \sqrt{ \frac{a_i}{a_j} } c_j \quad\textrm{ for all }\mv_k\in V_r \quad\textrm{ and all } i,j \in N(\mv_k).
\end{equation}
Moreover, the parameters $\tau_i$ are defined by
\begin{equation}\label{transl-t2}
\tau_0 = 0 \ \quad\textrm{ and } \ \tau_i = \sqrt{ \frac{a_i}{a_j} } \tau _j +\ell_j \quad\textrm{ for all } i,j \in N(\mv_k), \quad\textrm{ for all }\mv_k\in V_r.
\end{equation}
\end{theo}

\begin{proof}
As in the proof of Theorem \ref{thm_star},~\eqref{ccc-t} and \eqref{ps-t} imply~\eqref{larg-t} and \eqref{ampl-t} when $d_i=0$, and since all the vertices (and edges) are connected to each other, the propagation speeds are well-defined recursively, starting from $c_0$. Then,~\eqref{transl-t2} permits to compute all the $\tau_i$ starting from $\tau_0$ since our graph is a  tree. (Observe that even if the graph is infinite, the path from any given vertex to $v_0$ has certainly finite length.) Finally,~\eqref{cck-t} is equivalent to the Kirchhoff condition \eqref{Kirc} when we re-write $\frac{a_j}{c_j}$ in the form $\frac{a_{i_0}}{c_{i_0} b_{i_0} }  b_j$ for an arbitrarily chosen $i_0 \in N(\mv_k)$.
\end{proof}

\begin{rema}
If the graph is finite and therefore we are in the setting considered in~\cite{BC}, then  the wave constructed in Theorem \ref{thm_tree} is necessarily the unique solution of ${\rm(BBMG)}$, provided that the initial data $u_0$ is a solitary wave itself.
\end{rema}

\begin{rema}
In view of~\eqref{ampl-t} and \eqref{larg-t}, if all the coefficients $a_i$, $b_i$ and $d_i$ are respectively equal to $a$, $b$ and $d$, then all the propagation speeds agree with a common value $c$. Thus the Kirchhoff conditions \eqref{cck-t} are satisfied if and only if 
\[
 {\rm card}\; \{ i \ / \ \iota_{ki} =1 \} ={\rm card}\; \{ i/ \ \iota_{ki} =-1 \} \qquad \hbox{ for all }\mv_k \in V_r.
\]
In particular, recall this condition is satisfied (in fact, it is equivalent) to the directed graph being Eulerian (see e.g.~\cite[Thm.~4.4]{ChaLesZha10}) provided that $V_b=\emptyset$, i.e., that each vertex is of ramification type.
\end{rema}

\section{Networks with circuits}\label{scircuits}

\noindent In this section, we consider networks which contain circuits. We first treat the case of a graph having one directed circuit, i.e., a path linking a vertex $v \in V$ to itself following the incidence and having more than one edge (see Fig \ref{d-cir}). 

\begin{figure}[h]
\begin{center}
\includegraphics[width=0.6\textwidth]{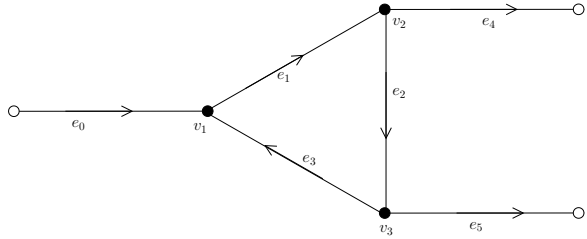}
\caption{A directed circuit.}\label{d-cir}
\end{center}
\end{figure}

\noindent As in Figure \ref{d-cir}, we denote by $\me_1, \ \me_2, \ \cdots, \ \me_n$  a directed path joining a vertex $\mv_1$ to itself, and let $\mv_1, \ \mv_2, \ \cdots, \ \mv_n$ the vertices of this path. We look for a solution $u$ of ${\rm(BBMG)}$ in the form $u_i(t,x)= \varphi_i(x-c_it +\tau_i)$ on each edge $\me_i$. According to~\eqref{prof1}, along the directed circuit, we have
\begin{equation}\label{cont-dc}
\varphi_{i+1}(z) = \varphi_i\left(\ell_i + \tau_i -\frac{c_i}{c_{i+1}}\tau_{i+1} +\frac{c_i}{c_{i+1}}z \right) \ \quad\textrm{ for all } \ z \in \mathbb{R},
\end{equation}
whence
\begin{eqnarray*}
\varphi_{i+2}(z)& =& \varphi_{i+1}\left(\ell_{i+1} + \tau_{i+1} -\frac{c_{i+1}}{c_{i+2}}\tau_{i+2} +\frac{c_{i+1}}{c_{i+2}}z \right)\\
& =& \varphi_i\left( \ell_i  +\frac{c_i}{c_{i+1}} \ell_{i+1} + \tau_i-\frac{c_i}{c_{i+2}} \tau_{i+2}+\frac{c_i}{c_{i+2}}z\right)  \quad\textrm{ for all } \ z \in \mathbb{R}
\end{eqnarray*}

\noindent By induction on the length of whole path we can thus prove that
\begin{equation}\label{per-dc}
\varphi_1(z) = \varphi_1\left(z+\sum_{i=1}^n \frac{c_1}{c_i}\ell_i \right) \ \quad\textrm{ for all } \ z \in \mathbb{R}.
\end{equation}
Thus any travelling wave is necessarily periodic and in particular we obtain:

\begin{lem}
If the graph contains a directed circuit, then there exist no solitary wave solutions of ${\rm(BBMG)}$.
\end{lem}

\noindent Next, we consider the case where the graph contains circuits, but no directed circuits (see Fig \ref{nd-cir}).

\begin{figure}[h]
\begin{center}
\includegraphics[width=0.8\textwidth]{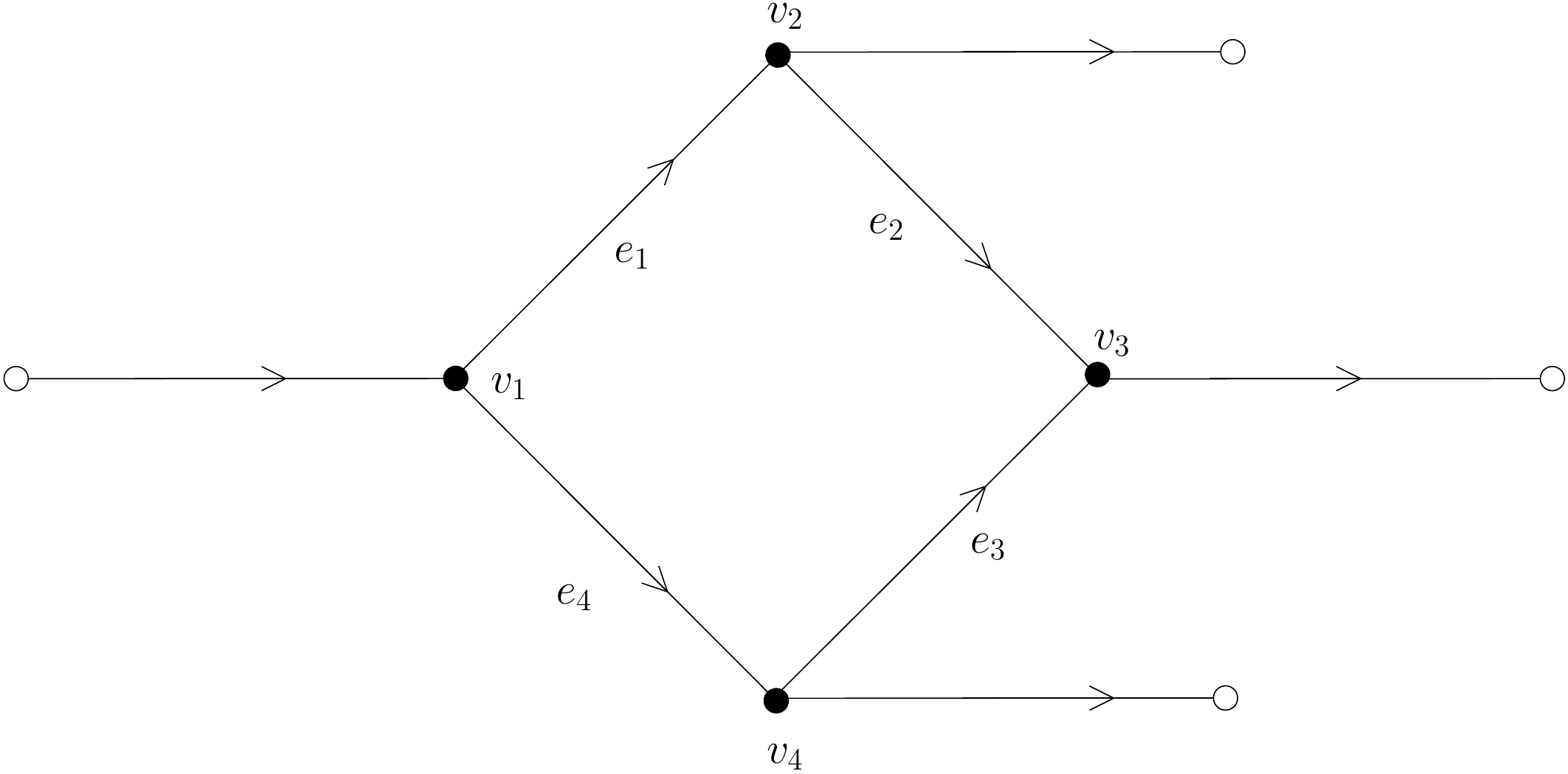}
\caption{A circuit which is not a directed circuit.}\label{nd-cir}
\end{center}
\end{figure}

It turns out that also on a graph containing undirected circuits a certain compatibility condition relating the lengths of the  different paths between two vertices has to be satisfied, in order for a travelling wave to exist. To begin with, we discuss the following simple example.

\begin{exa}\label{necess}
Let us begin by considering the simple case of the graph $G$ in Figure \ref{nd-cir}.  It is natural to address the following question: After splitting the incoming solitary wave in two waves at $\mv_1$ along the two paths $(\me_1,\me_2)$ and $(\me_4,\me_3)$, can we adjust the propagation speeds (and hence find suitable coefficients of $\rm(BBM)$) so that the two waves can eventually be glued to form one single outgoing wave at $\mv_3$?\\ In order to answer this question affirmatively we need to show that the hypotheses in Theorem \ref{thm_tree} and a compatibility condition stemming from \eqref{cont-dc} can be satisfied simultaneously. At vertex $\mv_3$, the continuity conditions along the path $(\me_1,\me_2)$ and the path $(\me_4,\me_3)$ and~\eqref{prof1} imply
\begin{equation}\label{cc-ndc}
l_1+ \frac{c_1}{c_2}l_2 = \frac{c_1}{c_4} \left(l_4+\frac{c_4}{c_3}l_3 \right).
\end{equation}
When we link $\mv_1$ to itself, from~\eqref{prof1}, we get
$$
\varphi_1(z) = \varphi_1(z + l_1+\frac{c_1}{c_2}l_2 - \frac{c_1}{c_3}l_3 - \frac{c_1}{c_4}l_4) \ \quad\textrm{ for all } \ z \in \mathbb{R},
$$
which is verified since $l_1+\frac{c_1}{c_2}l_2 =\frac{c_1}{c_3}l_3 + \frac{c_1}{c_4}l_4$ according \eqref{cc-ndc}. Linking $\mv_2$ to itself, we have
\begin{eqnarray}\label{contphi2}
\varphi_2(z) & = & \varphi_2(z - \frac{c_2}{c_1} l_1 - l_2 +\frac{c_2}{c_4}l_4 + \frac{c_2}{c_4}l_3 ) {} \nonumber \\ 
 & = & \varphi_2(z  - \frac{c_2}{c_1} \Big[ l_1 + \frac{c_1}{c_2 }l_2 -\frac{c_1}{c_4}l_4 - \frac{c_1}{c_4}l_3   \Big]) \quad\textrm{ for all } \ z \in \mathbb{R},
\end{eqnarray}
which is also satisfied because of \eqref{cc-ndc}. In the same way, linking $\mv_3$ and $\mv_4$ to themselves respectively,  the compatibility condition is verified if~\eqref{cc-ndc} is satisfied. \\
Using the definition of the propagtion speeds \eqref{ps-t}, we have 
$$
\frac{c_1}{c_j}= \prod_{i=1}^{j-1}  \frac{c_i}{c_{i+1}} = \prod_{i=1}^{j-1}  \sqrt{\frac{a_i}{a_{i+1}}} =  \sqrt{\frac{a_1}{a_j}} 
$$
and we can re-write~\eqref{cc-ndc} as.
\begin{equation}\label{compcond}
l_1+ \sqrt{\frac{a_1}{a_2}} l_2 = \sqrt{\frac{a_1}{a_4}} l_4 + \sqrt{\frac{a_1}{a_3}}l_3 =\sqrt{{a_1}} \Bigg( \frac{1}{\sqrt{a_4}} l_4 + \frac{1}{\sqrt{a_3}}l_3\Bigg).
\end{equation}
We conclude that~\eqref{compcond} is a necessary condition for the existence of a solitary wave on the graph $G$.
\end{exa}

\noindent In the general case, let us consider a graph $G$ with undirected circuits, but without any directed circuit. 

\begin{nota}
Let us denote $V_{\textrm{out}}$ the set of all vertex $\mv$ having at least two (directed) paths starting at $\mv$ and going to the same vertex $\mw$. For a vertex $\mv_i \in V_{\textrm{out}}$, we have at least two directed paths (whose lengths we denote by $n$ and $p-n$, respectively) ending at $\mw\in V$ along which the incidence factors are all equal to $1$. Gluing them, we obtain an (undirected)  path $(\me_{i_1},\ \me_{i_2},\ \cdots, \ \me_{i_n}, \ \me_{i_{n+1}},\  \cdots, \ \me_{i_p})$ with $1<n<p$ in $\mathbb{N}$ such that  $\mw = \me_{i_n} \cap \me_{i_{n+1}}$ and $\mv_i= \me_{i_1} \cap  \me_{i_p}$. We denote it by $[\mv_i, \mv_i]^n_p$.
\end{nota}

Along this undirected circuit, we compute a compatibility condition inspired by \eqref{cc-ndc} to satisfy the transmission conditions at $\mw$:
$$
\sum_{j=1}^n \frac{c_{i_1}}{c_{i_j}}l_{i_j} = \sum_{j=n+1}^p \frac{c_{i_1}}{c_{i_j}}l_{i_j},
$$
which can be re-written using conditions \eqref{ps-t}. Thus, we obtain the following.

\begin{theo}
Let~\eqref{simpl} hold. Suppose that the coefficients $a_i >0$ and $b_i \in \mathbb{R}^*$ satisfy the compatibility conditions \eqref{ccc-t} and \eqref{cck-t} for all $\mv_k\in V_r$ and all $i,j \in N(\mv_k)$. Then the following assertions hold.
\begin{enumerate}[(1)]
\item In order for a travelling wave solution to exist on $G$, the additional compatibility condition
\begin{equation}\label{comp_circuits}
\sum_{j=1}^n \frac{1}{\sqrt{a_{i_j}}}l_{i_j} = \sum_{j=n+1}^p \frac{1}{\sqrt{a_{i_j}}}l_{i_j} \ \quad\textrm{ for }\mv_i \in V_{out} \quad\textrm{ and all paths } [\mv_i, \mv_i]^n_p
\end{equation}
has to be satisfied.
\item Conversely, if~\eqref{comp_circuits} is satisfied, then there exists a solution $u$ of ${\rm(BBMG)}$ of the form 
\[
u_{|_{\me_i}}(t,x) = \varphi(x_i-c_it + \tau_i),
\]
where $\varphi$ is defined by~\eqref{smsw}, the propagation speeds are defined by
\begin{equation}\label{ps-ndc}
c_0 >0 \ \quad\textrm{ and } \ c_i = \sqrt{ \frac{a_i}{a_j} } c_j \quad\textrm{ for all }\mv_k\in V_r \quad\textrm{ and all } i,j \in N(\mv_k)
\end{equation}
and the parameters $\tau_i$ are defined by
\begin{equation}\label{Tau_cir}
\tau_0 = 0 \ \quad\textrm{ and } \ \tau_i = \sqrt{ \frac{a_i}{a_j} } \tau _j +l_j \quad\textrm{ for all }\mv_k\in V_r \quad\textrm{ and all } i,j \in N(\mv_k).
\end{equation}
\end{enumerate}
\end{theo}

\begin{proof}
(1) The claim can be proved by induction along the lines of the discussion in Example~\ref{necess} -- we omit the details.\\
(2) In order to prove the converse implication, we first observe that using conditions \eqref{ccc-t}, \eqref{cck-t} and \eqref{Tau_cir} we can construct as in Theorem~\ref{thm_tree} a wave satisfying ${\rm(BBMG)}$. We just have to check that there is no continuity jump in $u$ when considering circuits. Let us consider a circuit $[\mv_i, \mv_i]^n_p$ for some $\mv_i \in V_{out}$ and let $\mv = \me_{i_k} \cap \me_{i_{k+1}}$ be a vertex belonging to this circuit. We want to know whether $\varphi_k$ is well defined when we leave $\mv$ following the paths $(\me_{i_{k+1}},\ \cdots, \ \me_{i_p})$ and $(\me_{i_1},\ \cdots, \ \me_{i_k})$. Up to a relabeling we can assume that $k+1\leq n$.  As in \eqref{contphi2}, from 
$$
\varphi_{i_j}(\ell_{i_j} -c_{i_j}t + \tau_{i_j}) =\varphi_{i_{j+1}}( -c_{i_{j+1}}t + \tau_{i_{j+1}}) \ \quad\textrm{ for } 2\leq j+1 \leq n 
$$
and
$$
\varphi_{i_h}(-c_{i_h}t + \tau_{i_h}) =\varphi_{i_{h+1}}(\ell_{i_{h+1}}-c_{i_{h+1}}t + \tau_{i_{h+1}})
\quad\textrm{ for } n +2 \leq h+1 \leq p \ , 
$$ 
we obtain
$$
\varphi_{i_n}(z) = \varphi_{i_k} \left(\sum_{j=k}^{n-1} \frac{c_{i_k}}{c_{i_j}}\ell_{i_j} + \frac{c_{i_k}}{c_{i_n}}z + \tau_{i_k} -\frac{c_{i_k}}{c_{i_j}} \tau_{i_{i_n}} \right) \ ,
$$
$$
\varphi_{i_p}(z) = \varphi_{i_{n+1}}\left( -\sum_{j=n+2}^p \frac{c_{i_{n+1}}}{c_{i_j}}\ell_{i_j} + \tau_{i_{n+1}} - \frac{c_{i_{n+1}}}{c_{i_p}}\tau_{i_p} + \frac{c_{i_{n+1}}}{c_{i_p}}z \right)
$$
as well as
$$
\varphi_{i_k}(z) = \varphi_{i_1}\left(\sum_{j=1}^{k-1} \frac{c_{i_1}}{c_{i_j}}\ell_{i_j} + \tau_{i_1} -\frac{c_{i_1}}{c_{i_k}} \tau_{i_{i_k}}  + \frac{c_{i_1}}{c_{i_k}}z\right).
$$
Thus, using in particular the compatibility conditions
$$
\varphi_{i_n}(\ell_n-c_{i_n}t + \tau_{i_n}) =\varphi_{i_{n+1}}( \ell_{i{n+1}}-c_{i_{n+1}}t + \tau_{i_{n+1}}),
$$
and
$$
\varphi_{i_p}(-c_{i_p}t + \tau_{i_p}) =\varphi_{i_1}( -c_{i_1}t + \tau_{i_1})  ,
$$
that arise from~\eqref{eq:cont-eq} we are led to
\begin{eqnarray*}
\varphi_{i_k}(z) & = & \varphi_{i_k} \left( \sum_{j=k }^n \frac{c_{i_k}}{c_{i_j}}  \ell_{i_j} - \sum_{j=n+1}^p \frac{c_{i_k}}{c_{i_j}}  \ell_{i_j} + \sum_{j=1}^{k-1} \frac{c_{i_k}}{c_{i_j}}  \ell_{i_j} +z \right)\\ 
&=& \varphi_{i_k} \left( z + c_{i_k} \left[ \sum_{j=1 }^n \frac{1}{c_{i_j}}  \ell_{i_j} - \sum_{j=n+1}^p \frac{1}{c_{i_j}}  \ell_{i_j} \right] \right)
\end{eqnarray*}
Thanks to \eqref{comp_circuits} and by definition of the propagation speeds \eqref{ps-ndc}, this equation is satisfied and the wave is well defined along each circuit. The special case $k=n$, i.e. $\mv = \me_{i_n} \cap \me_{i_{n+1}}$, is the condition to glue back the waves into one single wave when leaving the circuit.
\end{proof}

\end{document}